%% file: infinite-rigidity.tex
\documentclass{amsart}
\usepackage{verbatim}
\usepackage[textsize=scriptsize]{todonotes}
\usepackage{longtable}
\usepackage{bbm}
\usepackage{amscd}
\usepackage{stmaryrd}

\usepackage[margin=1in]{geometry}
\input{Preamble.tex}

\let\scr=\mathcal
\def\cd{\mathrm{cd}}

\def\comp{\wedge}

\newcommand{\SH}{\mathcal{SH}}
\newcommand{\Spc}{\mathcal{S}\mathrm{pc}}
\newcommand{\et}{{\acute{e}t}}
\newcommand{\ret}{{r\acute{e}t}}
\newcommand{\wequi}{\simeq}
\newcommand{\1}{\mathbbm{1}}

\newcommand{\Gm}{{\mathbb{G}_m}}

\def\PSh{\mathcal{P}}

\def\ph{\mathord-}
\def\op{\mathrm{op}}

\def\Cat{\mathcal{C}\mathrm{at}{}}

\def\CAlg{\mathrm{CAlg}}

\def\Sm{{\mathcal{S}\mathrm{m}}}
\def\Sch{{\mathcal{S}\mathrm{ch}}}

\newcommand{\Mod}{\mathrm{Mod}}

\DeclareMathOperator{\id}{\mathrm{id}}
\def\Et{\mathcal{E}\mathrm{t}}
\def\mot{\mathrm{mot}}

\newtoggle{final}
\toggletrue{final}

\iftoggle{final} {
\renewcommand{\todo}[1]{}
\newcommand{\NB}[1]{}
}{ 
\newcommand{\NB}[1]{\todo[color=gray!40]{#1}}
}

\title{Remarks on étale motivic stable homotopy theory}
\date{\today}

\author{Tom Bachmann}
\address{Mathematischen Institut, LMU Munich, Germany}
\email{tom.bachmann@zoho.com}

\makeatletter
\let\@wraptoccontribs\wraptoccontribs
\makeatother

\contrib[with an appendix by]{Tom Bachmann and Marc Hoyois}

\begin{document}

\maketitle

\setcounter{tocdepth}{1}
\tableofcontents

\section{Introduction}
We strengthen some results in étale (and real étale) motivic stable homotopy theory, by eliminating finiteness hypotheses, additional localizations and/or extending to spectra from $H\Z$-modules.
The main results are Theorem \ref{thm:infinite} (proving rigidity in $\ell$-adic étale motivic stable homotopy theory for any $\Z[1/\ell]$-scheme), Theorem \ref{thm:real} (proving rigidity in real étale motivic stable homotopy theory for any scheme) and Theorem \ref{thm:marc} (proving $p$-periodicity in étale motivic stable homotopy theory of characteristic $p$-schemes).

\subsection*{Notation and conventions}
Given a scheme $S$, we denote by $S_\et^\comp$ its hypercompleted étale $\infty$-topos and by $\Sp(S_\et^\comp)$ the stabilization thereof.
In contrast, we denote by $\SH_\et^\comp(S)$ the localization of the motivic stable category $\SH(S)$ at the (desuspended) étale hypercovers, and similarly denote by $\SH_\et^{\comp S^1}(S)$ the localization of motivic $S^1$-spectra.
We also write $\Sp(\Sm_{S,\et})$ for the category of spectral étale sheaves on the site $\Sm_S$, and $\Sp(\Sm_{S,\et}^\comp)$ for its hypercompletion.
Given a stable $\infty$-category $\scr C$, we denote by $\scr C_p^\comp$ the subcategory of $p$-complete objects.

We freely use the language and notation of \cite{HTT,HA}.

\section{Invertible retracts of symmetric objects}
\begin{prop} \label{prop:inv-retract}
Let $\scr C$ be a symmetric monoidal $\infty$-category, $I \in \scr C$ invertible and $I \xrightarrow{e} X \xrightarrow{f} I$ a retraction.
If $X$ is symmetric (i.e. the cyclic permutation on $X^{\otimes n}$ is homotopic to the identity for some $n \ge 2$) then $e$ and $f$ are inverse equivalences.
\end{prop}
\begin{proof}
Since a $(2n-1)$-cycle is a product of two $n$-cycles, we may assume that $X$ is $n$-symmetric where $n$ is odd\NB{More generally, subgroup of $S_n$ generated by all $k$-cycles for some $k \ge 2$ is $A_n$ or $S_n$ depending on parity of $k$. Hence given $m \ge n$, $n$-symmetric implies $m$-symmetric if $m$ is odd or $n$ is even. Also for invertible objects, symmetric is the same as $3$-symmetric.}; then also $I^{-1}$ is $n$-symmetric \cite[Lemma 4.17]{dugger2014coherence}.
It follows that $X \otimes I^{-1}$ is $n$-symmetric, whence (replacing $X$ by $X \otimes I^{-1}$) we may assume that $I = \1$.

First we offer a slightly simpler proof in the case that $\scr C$ is semiadditive and idempotent complete, and the tensor product distributes over sums.
In this case we can write $X \wequi \1 \oplus X'$, and so $X^{\otimes n} \wequi \1 \oplus X'^{\oplus n} \oplus \dots$, where the symmetric group action restricts to the summand $X'^{\oplus n}$ and yields the canonical action for the $\oplus$ symmetric monoidal structure.
In particular $X'$ is $\oplus$-symmetric, which implies that $1=0$ as endomorphisms of $X'$, i.e. $X'=0$ as desired.

In general, consider the maps \[ e': X \wequi X \otimes \1^{\otimes n-1} \xrightarrow{\id \otimes e^{\otimes n-1}} X^{\otimes n} \quad\text{and}\quad f': X^{\otimes n} \xrightarrow{\id \otimes f^{\otimes n-1}} X \otimes \1^{\otimes n-1} \wequi X. \]
Then $f'e' \wequi \id \otimes \id^{\otimes n-1} \wequi \id$.
On the other hand if $\sigma$ is the cyclic permutation of $X^{\otimes n}$, then $f'\sigma e'$ is the tensor product of $f: X \to \1$, $n-2$ copies of $\id: \1 \to \1$ and $e: \1 \to X$, which (up to unit equivalences) is the same as $ef$.
Consequently if $\sigma \wequi \id$ then $\id \wequi f'e' \wequi f'\sigma e' \wequi ef$, so that $e,f$ are indeed inverse equivalences.
\end{proof}

\begin{exm}
Since invertible objects are symmetric \cite[Lemma 4.17]{dugger2014coherence}, Proposition \ref{prop:inv-retract} strengthens \cite[Lemma 30]{bachmann-real-etale}.
\end{exm}
\begin{exm} \label{exm:Gm-stably-symm}
Since $S^{2,1} \in \Spc(S)_*$ is $3$-symmetric (see e.g. \cite[Lemma 6.3]{hoyois-equivariant}), $\Gm \wequi S^{-1} \wedge S^{2,1} \in \SH^{S^1}(S)$ is also $3$-symmetric.
\end{exm}

\section{Étale topology} \label{sec:etale}
The following result was proved in \cite[Theorem 6.5]{bachmann-SHet} under additional finiteness assumptions, and with an additional localization on the middle category.
\begin{thm} \label{thm:infinite}
Let $S$ be a scheme with $\ell \in \scr O_S^\times$.
Then the canonical functors \[ \Sp(S_\et^\comp)_\ell^\comp \to \SH_\et^{\comp S^1}(S)_\ell^\comp \to \SH_\et^\comp(S)_\ell^\comp \] are equivalences.
\end{thm}

\newcommand{\DA}{D\!A}
We can also remove finiteness assumptions from related rigidity results.
For example, denote by $D(S_\et^\comp, \Z/\ell)$ the unbounded derived category sheaves of $\Z/\ell$-modules, and by $\DA_\et^\comp(S, \Z/\ell)$ Ayoub's category of étale motives without transfers \cite[\S3]{ayoub2014realisation}.
The equivalence of the outer two categories in the following result was proved in \cite[Theoreme 4.1]{ayoub2014realisation} under additional finiteness assumptions.
\begin{cor}
Let $S$ be a scheme with $\ell \in \scr O_S^\times$.
Then the canonical functors \[ D(S_\et^\comp, \Z/\ell) \to \DA_\et^{\comp S^1}(S, \Z/\ell) \to \DA_\et^\comp(S, \Z/\ell) \] are equivalences.
\end{cor}
\begin{proof}
Writing $H\Z/\ell \in \CAlg(\Sp(S_\et^\comp)_\ell^\comp)$ for the Eilenberg--MacLane spectrum, the result follows from the equivalences $D(S_\et^\comp, \Z/\ell) \wequi \Mod_{H\Z/\ell}(\Sp(S_\et^\comp)_\ell^\comp)$, $\DA_\et^{\comp S^1}(S, \Z/\ell) \wequi \Mod_{H\Z/\ell}(\SH_\et^{\comp S^1}(S)_\ell^\comp)$ and $\DA_\et^\comp(S, \Z/\ell) \wequi \Mod_{H\Z/\ell}(\SH_\et^\comp(S)_\ell^\comp)$.
The third is a formal consequence of the second, and the first two follow from \cite[Theorem 2.1.2.2]{SAG}.
\end{proof}

We recall some ingredients used in the proof of Theorem \ref{thm:infinite}.

(0) If $F: \scr C \to \scr D$ is a cocontinuous, symmetric monoidal functor of presentably symmetric monoidal $\infty$-categories and $\scr C$ is $\ell$-complete (i.e. $\scr C \wequi \scr C_\ell^\comp$), then so is $\scr D$.
Indeed for objects $X, Y \in \scr D$, the functor $F(\ph) \otimes X: \scr C \to \scr D$ admits a right adjoint $r_X$, and hence $\Map_{\scr D}(X, Y) \wequi \Map_{\scr C}(\1, r_X Y)$ is $\ell$-complete as needed.
In particular, $\ell$-completion commutes with localization of presentably symmetric monoidal $\infty$-categories.

(1) The functors \[ \Sp(S_\et^\comp)_\ell^\comp, \SH_\et^{\comp S^1}(S)_\ell^\comp, \SH_\et^\comp(S)_\ell^\comp: \Sch_{\Z[1/\ell]}^\op \to \Cat_\infty \] are Zariski sheaves.
This implies that they are right Kan extended from their restriction to affine $\Z[1/\ell]$-schemes (see e.g. \cite[Lemma C.3]{hoyois2015quadratic}), which is what we shall use.
The descent properties are established by arguments entirely analogous to e.g. \cite[Proposition 4.8]{hoyois-equivariant} \cite[\S2.3]{ayoub2020six}; in fact all three functors satisfy étale hyperdescent.

(2) The functor \[ \Et_{(\ph)}^\mathrm{fp}: \Sch^\op \to \Cat_\infty \] (sending $X$ to the category of finitely presented étale $X$-schemes) is \emph{continuous}: it converts cofiltered limits of quasi-compact quasi-separated schemes with affine transition maps into colimits \cite[Lemme VII.5.6]{sga4}.
This implies that also \[ \PSh(\Et_{(\ph)}^\mathrm{fp}): \Sch^\op \to Pr^L \] is continuous.
From this one deduces the same result for spectral presheaves.
The category of étale \emph{sheaves} is obtained by inverting the nerves of (finitely presented) étale covers, which must be pulled back from a finite stage by continuity of $\Et_{(\ph)}$ and quasi-compactness.
It follows that \[ \Sp((\ph)_\et): \Sch^\op \to Pr^L \] is continuous.
Beware the absence of hypercompletion!
By similar arguments, the functor \[ \SH_\et^{S^1}(\ph): \Sch^\op \to Pr^L \] is continuous (again no hypercompletion).
From this one easily deduces (e.g. using (0)) that also $\Sp((\ph)_\et)_\ell^\comp, \SH_\et^{S^1}(\ph)_\ell^\comp$ are continuous.

(3) Finally recall the object $\hat\1_\ell(1)[1] \in \Sp(S_\et^\comp)_\ell^\comp$ and the map $\sigma: \Gm \to \hat\1_\ell(1)[1] \in \SH_\et^{\comp S^1}(S)_\ell^\comp$ from \cite[\S3]{bachmann-SHet}.

\begin{proof}[Proof of Theorem \ref{thm:infinite}.]
We first show that $\sigma: \Gm \to \hat\1_\ell(1)[1] \in \SH_\et^{\comp S^1}(S)_\ell^\comp$ is an equivalence.
By stability of $\sigma$ under base change, we reduce to $S=Spec(\Z[1/\ell])$, and by \cite[Corollary 5.12]{bachmann-SHet} we reduce to $S=Spec(k)$, where $k$ is a separably closed field.
In this situation $\sigma: \Gm \to \hat\1_\ell(1)[1]$ admits a section \cite[proof of Theorem 6.5]{bachmann-SHet}, and thus $\sigma$ is an equivalence by Proposition \ref{prop:inv-retract} and Example \ref{exm:Gm-stably-symm}.

We have thus proved that $\SH_\et^{\comp S^1}(S)_\ell^\comp \wequi \SH_\et^\comp(S)_\ell^\comp$.
In particular the theorem holds whenever \cite[Theorem 6.6]{bachmann-SHet} applies, so for example $(*)$ if $S$ is finite type over $\Z[1/\ell]$ or the spectrum of a separably closed field of characteristc $\ne \ell$.

To prove that $\Sp(S_\et^\comp)_\ell^\comp \wequi \SH_\et^{\comp S^1}(S)_\ell^\comp$, we may assume by right Kan extension that $S$ is affine.
In this case we can write $S = \lim_\alpha S_\alpha$, where each $S_\alpha$ is affine (so in particular quasi-compact quasi-separated) and of finite type over $\Z[1/\ell]$.
Denote by $V_\alpha$ (respectively $V$) the class of $\infty$-connective maps in $\Sp(S_{\alpha,\et})$ (respectively $\Sp(S_{\et})$) and by $W_\alpha$ (respectively $W$) the $\infty$-connective maps in $\Sp(\Sm_{S_\alpha,\et})$ (respectively $\Sp(\Sm_{S,\et})$).
Then $\Sp(S_{\alpha,\et}^\comp)_\ell^\comp \wequi \Sp(S_{\alpha,\et})_\ell^\comp[V_\alpha^{-1}]$ (in other words, $\Sp(S_{\alpha,\et}^\comp)_\ell^\comp$ is the initial object of $Pr^L_{\Sp(S_{\alpha,\et}^\comp)_\ell^\comp/}$ inverting $V_\alpha$), and similarly $\SH_\et^{\comp S^1}(S_\alpha)_\ell^\comp \wequi \SH_\et^{S^1}(S_\alpha)_\ell^\comp[W_\alpha^{-1}]$.
By continuity, $\Sp(S_\et)_\ell^\comp \wequi \colim_\alpha \Sp(S_{\alpha,\et})_\ell^\comp$.
Consider the commutative diagram in $Pr^L$
\begin{equation*}
\begin{CD}
\Sp(S_\et)_\ell^\comp @>a>> \colim_\alpha \Sp(S_{\alpha,\et}^\comp)_\ell^\comp[V_\alpha^{-1}] @>b>> \Sp(S_\et^\comp)_\ell^\comp \\
@VVV                              @VcVV                                                          @VdVV                 \\
\SH_\et^{S^1}(S)_\ell^\comp @>a'>> \colim_\alpha \SH_\et^{S^1}(S_\alpha)_\ell^\comp[W_\alpha^{-1}] @>b'>> \SH_\et^{\comp S^1}(S)_\ell^\comp.
\end{CD}
\end{equation*}
The morphism $a$ is a localization (namely at the union of the pullbacks to $S$ of the classes $V_\alpha$), and $ba$ is a localization (namely at $V$); hence so is $b$.
Similarly $b'$ is a localization.
The morphism $c$ is an equivalence, being a colimit of equivalences by $(*)$.
We deduce that $d$ is a localization.
To conclude the proof, it thus suffices to prove that $d$ is conservative.
Since equivalences of hypercomplete sheaves may be tested on stalks (essentially by definition of hypercompleteness, this reduces to the well-known special case of sheaves of abelian groups), an object of $\Sp(S_\et^\comp)_\ell^\comp$ vanishes if and only if its image under the pullback to $\Sp(\bar s_\et^\comp)_\ell^\comp$ vanishes, for every geometric point $\bar s = Spec(\bar k) \to S$ (where $\bar k$ is a separably closed field).
Since formation of $d$ is natural in $S$, we are reduced to the case $S=\bar s$, which was already established (see $(*)$).
\end{proof}

\todo{is $\SH_\et^{\comp S^1}(S) \to \SH_\et^\comp(S)$ fully faithful?}

\begin{cor}
Let $S$ be a scheme with $\ell \in \scr O_S^\times$.
Then $\Sp(S_\et^\comp)_\ell^\comp \to \Sp((S \times \A^1)^\comp_\et)_\ell^\comp$ is fully faithful.
(In other words, ``$\ell$-adic hyper-étale cohomology with spectral coefficients is $\A^1$-invariant''.)
\end{cor}
\begin{proof}
This holds for $\SH_\et^\comp(S)_\ell^\comp$ essentially by construction.
\end{proof}

\section{Real étale topology}
Recall the real étale topology from \cite{real-and-etale-cohomology}.
Denote by $S_\ret$ the small real étale $\infty$-topos of $S$ (not hypercompleted), by $\SH_\ret(S)$ the localization of $\SH(S)$ at the real étale covers, and so on.
\begin{rmk}
If $\dim S < \infty$, then $S_\ret$ and $\Sm_{S,\ret}$ are hypercomplete \cite[Theorem B.13]{elmanto2019scheiderer}.
It follows that in this situation, our notation coincides with the one from \cite{bachmann-real-etale} (where everything is hypercompleted and finite dimensional by definition).
\end{rmk}
The following result strengthens \cite[Theorem 35]{bachmann-real-etale}, by removing $\rho$-inversion from $\SH^{S^1}_\ret(S)$ and finiteness assumptions from $S$.
\begin{thm} \label{thm:real}
Let $S$ be any scheme.
Then \[ \SH(S)[\rho^{-1}] \wequi \SH_\ret(S) \wequi \SH^{S^1}_\ret(S) \wequi \Sp(S_\ret). \]
In particular $\rho: S^0 \to \Gm \in \SH^{S^1}_\ret(S)$ is an equivalence.
\end{thm}
\begin{proof}
By Zariski descent and continuity (see \S\ref{sec:etale} and \cite[proof of Proposition 3.4.1]{real-and-etale-cohomology}), we may assume that $S$ is finite type over $\Z$.
In this case by \cite[Theorem 35]{bachmann-real-etale}, only the last statement requires proof.
The proof of \cite[Proposition 29]{bachmann-real-etale} constructs a retraction $S^0 \xrightarrow{\rho} \Gm \to S^0 \in \Spc_{ret}(S)_*$.
The result thus follows from Proposition \ref{prop:inv-retract} and Example \ref{exm:Gm-stably-symm}.
\end{proof}

\appendix
\section{Vanishing of $\SH_\et^\comp(\F_p)_p^\comp$}

\begin{thm} \label{thm:marc}
We have $\1/p \wequi 0 \in \SH^{\comp S^1}_\et(\F_p)$.
In particular if $X \in \Sch_{\F_p}$ then \[ \SH_\et^{\comp S^1}(X)_p^\comp = * = \SH_\et^\comp(X)_p^\comp. \]
\end{thm}

Before the proof, we need some preparation.
The category $\Sp(\Sm_{S,\et}^\comp)$ of étale hypersheaves of spectra on $\Sm_S$ admits a canonical non-degenerate $t$-structure (see e.g. \cite[\S2.2]{bachmann-SHet}).
Denote by $L_{\et,\mot}^\comp$ the localization endofunctor of $\Sp(\Sm_{S,\et}^\comp)$ corresponding to the $\A^1$-equivalences, so that the category of local objects is $\SH_\et^{\comp S^1}(S)$.
\begin{lem} \label{lem:lmot-conn}
If $E \in \Sp(\Sm_{\F_p,\et}^\comp)_{\ge 0}$, then $L_{\et,\mot}^\comp E/p \in \Sp(\Sm_{\F_p,\et}^\comp)_{\ge -1}$.\NB{In fact $E/p \in \Sp(\Sm_{\F_p,\et}^\comp)_{\ge 0}$.}
\end{lem}
\begin{proof}
Denote by $L_{\A^1} E$ the presheaf \[ X \mapsto \colim_{n \in \Delta^\op} E(X \times \A^n). \]
Then $L_{\A^1} E$ is $\A^1$-invariant and $E \to L_{\A^1} E$ is an $\A^1$-equivalence \cite[Corollaries 2.3.5 and 2.3.8]{A1-homotopy-theory}.
Moreover since $\cd(\F_p) < \infty$, étale hypersheaves are closed under colimits in presheaves (see e.g. \cite[Lemma 2.16]{bachmann-SHet}), and thus $L_{\et,\mot}^\comp E \wequi L_{\A^1} E$.
Since $\Sp_{\ge -1}$ is closed under colimits, it thus suffices to show that for $X \in \Sm_{\F_p}$ affine we have $(E/p)(X) \in \Sp_{\ge -1}$.
This follows from \cite[Lemma 2.7(2)]{bachmann-SHet}, using that affine $\F_p$-schemes have $p$-étale cohomological dimension $\le 1$ \cite[Théorème X.5.1]{sga4}.
\end{proof}

\begin{proof}[Proof of Theorem \ref{thm:marc}.]
Only the first statement requires proof.
Since $\cd(\F_p) < \infty$, $\SH_\et^{\comp S^1}(\F_p)$ is compactly generated by representables \cite[Corollary 5.7]{bachmann-SHet} and thus $L_{\et,\mot}^\comp: \Sp(\Sm_{S,\et}^\comp) \to \Sp(\Sm_{S,\et}^\comp)$ preserves colimits.
Let $H\Z \in \Sp(\Sm_{S,\et}^\comp)_{\ge 0}$ denote the Eilenberg--MacLane spectrum.
We seek to prove that $L_{\et,\mot}^\comp(\1/p) = 0$.
By Lemma \ref{lem:lmot-conn} we have $L_{\et,\mot}^\comp(\1/p) \in \Sp(\Sm_{S,\et}^\comp)_{\ge -1}$, and hence $L_{\et,\mot}^\comp(\1/p) = 0$ if and only if $L_{\et,\mot}^\comp(\1/p) \wedge H\Z = 0$.
Since $L_{\et,\mot}^\comp$ preserves colimits and $H\Z$ lies in the subcategory generated under colimits by $\1$, we have $L_{\et,\mot}^\comp(\1/p) \wedge H\Z \wequi L_{\et,\mot}^\comp(H\Z/p)$.
The forgetful functor $U: \Mod_{H\Z}(\Sp(\Sm_{S,\et}^\comp)) \to \Sp(\Sm_{S,\et}^\comp)$ commutes with $L_{\et,\mot}^\comp$ (in fact $L_{\et,\mot}^\comp$ is given by $L_{\A^1}$ in both categories, see the proof of Lemma \ref{lem:lmot-conn}, and $U$ preserves colimits \cite[Corollary 4.2.3.5]{HA}), and the motivic localization of $\Mod_{H\Z}(\Sp(\Sm_{S,\et}^\comp))$ is $\DA_\et^{\comp S^1}(\F_p, \Z)$ (use \cite[Theorem 2.1.2.2]{SAG}).
Consequently we have reduced to proving that $\DA_\et^{\comp S^1}(\F_p, \Z/p) = *$, or equivalently that the unit of this symmetric monoidal category vanishes.
This works using the standard argument, i.e. the fiber sequence $\Z/p \to \mathbb{G}_a \to \mathbb{G}_a$.
\end{proof}

\bibliographystyle{alpha}
\bibliography{bibliography}

\end{document}

%% file: Preamble.tex
\usepackage{verbatim}
\usepackage[textsize=scriptsize]{todonotes}
\usepackage{tikz-cd}
\usepackage{etoolbox}
\usepackage{etex}
\usepackage[pdfusetitle,unicode,hidelinks]{hyperref}
\usepackage{cleveref}
\usepackage[T1]{fontenc}

\usepackage{chemarr}
\usepackage{amssymb}
\usepackage{amsmath}
\usepackage{comment}
\usepackage{spectralsequences}
\usepackage{cleveref}
\usepackage{mathtools}
\usepackage{rotating}
\usepackage{wrapfig}
\usepackage{outlines}
\usepackage{graphicx}

\numberwithin{equation}{section}

\theoremstyle{definition}

\newtheorem{rmk}[equation]{Remark}

\newtheorem{exm}[equation]{Example}

\newtheorem*{dfn*}{Definition}
\newtheorem*{axm*}{Axiom}
\newtheorem*{ntn*}{Notation}
\newtheorem*{exm*}{Example}
\newtheorem*{exr*}{Exercise}
\newtheorem*{int*}{Intuition}
\newtheorem*{qst*}{Question}
\newtheorem*{rmk*}{Remark}

\theoremstyle{plain}

\newtheorem{thm}[equation]{Theorem}
\newtheorem{prop}[equation]{Proposition}
\newtheorem{lem}[equation]{Lemma}

\newtheorem{cor}[equation]{Corollary}

\newtheorem*{thm*}{Theorem}
\newtheorem*{prop*}{Proposition}
\newtheorem*{cor*}{Corollary}
\newtheorem*{lem*}{Lemma}
\newtheorem*{cnj*}{Conjecture}

\let\oldwidetilde\widetilde
\protected\def\widetilde{\oldwidetilde}

\DeclareMathOperator*{\colim}{\mathrm{colim}}
\DeclareMathOperator{\Map}{\mathrm{Map}}

\newcommand{\F}{\mathbb{F}}

\newcommand{\A}{\mathbb{A}}

\DeclareMathOperator{\Sp}{\mathrm{Sp}}

\newcommand{\Z}{\mathbb{Z}}

\DeclarePairedDelimiter\abs{\lvert}{\rvert}%
\makeatletter
\let\oldabs\abs
\def\abs{\@ifstar{\oldabs}{\oldabs*}}

\usepackage{tikz}
\usetikzlibrary{matrix,arrows,decorations}
\usepackage{tikz-cd}

\usepackage{adjustbox}

\let\oldtocsection=\tocsection
 
\let\oldtocsubsection=\tocsubsection
 
\let\oldtocsubsubsection=\tocsubsubsection
 
\renewcommand{\tocsection}[2]{\hspace{0em}\oldtocsection{#1}{#2}}
\renewcommand{\tocsubsection}[2]{\hspace{1em}\oldtocsubsection{#1}{#2}}
\renewcommand{\tocsubsubsection}[2]{\hspace{2em}\oldtocsubsubsection{#1}{#2}}
